\documentclass[10pt]{article}
\usepackage[dvips]{graphicx}
\usepackage{amssymb}
\usepackage{color} 
\usepackage{amsmath}
\usepackage{amsthm}

\newtheorem{theorem}{Theorem}[section]
\def \S {{\cal S}}
\def \A {{\cal A}}
\def \U {{\cal U}}
\def \B {{\cal B}}

\def \deg {{\rm deg}}

\newtheorem{conjecture}[theorem]{Conjecture}
\newtheorem{corollary}[theorem]{Corollary}
\newtheorem{proposition}[theorem]{Proposition}
\newtheorem{definition}[theorem]{Definition}

\newtheorem{acknowledgements}[theorem]{Acknowledgements}
\begin{document}

\title{Composition of permutation representations of triangle groups}
\author{ Siddiqua Mazhar} 
 
\date{}
\maketitle 

\begin{itemize}

\item[] Department of Mathematics and Statistics, Newcastle University, \\United Kingdom. \thanks{mazharsiddiqua@gmail.com}     
\end{itemize}


\begin{abstract}
A triangle group is denoted by $\Delta(p,q,r)$ and has finite presentation
$$ \Delta(p,q,r)=\langle x,y | x^p=y^q=(xy)^r=1 \rangle .$$
We examine a method for composition of permutation representations of 
a triangle group $\Delta(p,q,r)$ that was used in Everitt's proof of Higman's
1968 conjecture that every Fuchsian group has amongst its homomorphic images all
but finitely many alternating groups.
We see that some of these compositions must give imprimitive representations,
and in particular situations, where the representations being composed are
all equivalent copies of an alternating group in the same degree, we can give a 
complete description of the structure of the composition of the representations.
This article contains the main results of the author's PhD thesis.
\end{abstract}

\vskip20pt
\noindent
{\it Mathematics Subject Classifications:} 20XX, 05C25. 

\noindent
 {\it Keywords:} Triangle groups, coset diagram,  Imprimitive group.

\section{Introduction}
We examine a method for composition of permutation representations of 
a triangle group $\Delta(p,q,r)$ that was used in the proof \cite{Everitt2} of Higman's
conjecture, that every Fuchsian group has amongst its homomorphic images all
but finitely many alternating groups.
We see that some of these compositions must give imprimitive representations,
and in particular situations, where the representations being composed are
all equivalent copies of an alternating group in the same degree, we can give a 
complete description of the structure of the composition of the representations.
This article contains the main results of the PhD thesis \cite{Siddiqua}.

This first section is introductory, and the second section contains definitions and notation. Section~\ref{sec:Comp} describes in Theorem~\ref{thm:comp}
the construction that we use for composing representations, and Section~\ref{sec:ImprimComp} in Theorems~\ref{thm:imprim_t=p}, \ref{thm:imprim_t=1}, \ref{thm:imprim_comp}
three variants of this that give imprimitive representations. Finally in Section~\ref{sec:Alt}, Theorems~\ref{thm:ImprimCompAlt1} and ~\ref{thm:ImprimCompAlt2}  examines the structure of the composed representations when they are build out of equivalent 
representations, that are all alternating groups in the same degree.

\section{Definitions and notations}
We define the triangle group $\Delta(p,q,r)$ to be the group
\[ \Delta(p,q,r) = \langle x,y \mid x^p=y^q=(xy)^r=1 \rangle. \]
We refer to \cite{Cameron} for standard definitions and notation for permutation groups.
In particular, we shall use the following notation.
$\S(\Omega)$ will denote the group of all permutations on a set $\Omega$ and
$\A(\Omega)$ the subgroup of all even permutations (acting on the right). Further, 
$\S_n$ and $\A_n$ will denote the groups of all permutations and even permutations, respectively, of the set $\{1,2,\ldots,n\}$. We use the notation $x^y$ to denote the conjugate $y^{-1}xy$, as is consistent for right actions. 

This article studies transitive permutation representations of triangle groups on
finite sets $\Omega$, and in particular the composition of representations described in \cite{Everitt2}. It is often convenient to view such representations graphically, using \textbf{coset diagrams}.
We define a coset diagram for a triangle group $G=\Delta(p,q,r)$ to be a graphical description of
a permutation representation of $G$ on a set $\Omega$, in which the elements of $\Delta$ are
represented as vertices, and directed edges labelled $x$ and $y$ join a vertex
$g$ to its images under $x$ and $y$.

Where a permutation group $H$ acts on set $\Omega$, we recall the definition of 
a \textbf{block} for $H$ as a subset $X$ of $\Omega$ for which either $X^h=X$ 
or $X^h \cap X=\emptyset$, for each $h \in H$, and call the action of $H$
\textbf{primitive} if $\emptyset$ and $\Omega$ are the only blocks,
\textbf{imprimitive} otherwise. 
For an imprimitive action, the partition of $\Omega$ into a set of blocks is 
called a \textbf{block system}.

The following result is important for the results of this paper. It is found as (essentially) \cite[Theorem 1.8]{Cameron},
which refers to \cite{Suprunenko} for a proof. 

\begin{proposition} \label{lemma:ImprimWreathProduct}
Let $H \subseteq \S(\U)$ act transitively and imprimitively on a set $\U$ 
with block system $\B = \{ B_1, \ldots, B_\deg\}$,
and let $\Omega = \{1, \ldots, \deg\}$. 

Let $J_1,\ldots,J_\deg \subseteq H$ be the setwise stabilisers of the blocks
$B_1,\ldots,B_\deg$ and let $K_1,\ldots,K_\deg$ be the pointwise stabilisers of
those blocks.  Let $\psi : H \rightarrow \S_\deg$ define
the action of $H$ on the block system $\B$ and 
let $N=\cap_{i=1}^\deg J_i$ be the kernel of $\psi$.
  
Let $Q_i \subseteq \S(\U)$ be the group of permutations of $B_i$ defined
by the action of $J_i$ on $B_i$ (so $Q_i \cong J_i/K_i$),
and fixing all points of $\U \setminus B_i$,
and let $P_i \subseteq Q_i \subseteq  \S(\U)$ be the
group of permutations of $B_i$ defined by the action of $N$ on $B_i$
(so $P_i \cong N/N \cap K_i$).

Then the groups $Q_i$ are all isomorphic to a single group $Q$, and pairwise
commute. Further
$H$ is isomorphic to a subgroup of
\[ Q_1Q_2\cdots Q_\deg \rtimes \psi(H) 
\cong Q^\deg \rtimes \psi(H),\] that is, $H$ is isomorphic to a subgroup of
$H \subseteq Q \wr \psi(H)$,
and  $N$ is isomorphic to a subgroup of $P_1\cdots P_\deg$.
\end{proposition}

\section{Composition}
\label{sec:Comp}

We describe a method to compose $t\leq p$ coset diagrams for a triangle group $\Delta(p,q,r)$ to get a transitive diagram. 

\begin{definition}
In an arbitrary permutation representation $\pi$ of $G$, two points $a$ and $b$, which are fixed by $\pi(x)$ and such that $(\pi(x)\pi(y))^k$ map $a$ to $b$, form a \textbf{$k$-handle} and are denoted by $[a,b]_k$ \cite{Conder2}. 
\end{definition}

The construction in the theorem below is described in
Proposition 3.1 in \cite{Everitt2}. We provide a new (algebraic) proof of the correctness of this construction.

\begin{theorem}\label{thm:comp}
Let $G=\Delta(p,q,r)$ be a triangle group, and 
let $t \leq p$ be a positive integer. 
Suppose that $G$ has transitive permutation
representations $\pi_1,\ldots,\pi_t$ on distinct, disjoint, finite sets $\Omega_1,\ldots,\Omega_t$,
and that $[a_1,b_1],\ldots,[a_p,b_p]$ are all $k$-handles, with $[a_i,b_i]$
a handle in $\Omega_{j_i}$. If $j_i=j_{i'}$ for for some $i \neq i'$, then
suppose that the handles $[a_i,b_i]$ and $[a_{i'},b_{i'}]$ are disjoint.
Now define permutations $\phi_x$, $\phi_y$ of the disjoint union $\U$ of $\Omega_1,\ldots,\Omega_t$
via
\begin{eqnarray*}
\phi_x &=& \pi_1(x)\cdots \pi_t(x) \circ (a_1,\cdots,a_p)(b_b,b_{p-1},\cdots, b_1),\\
\phi_y &=& \pi_1(y)\cdots \pi_t(y). \\
\end{eqnarray*}
Then $\phi_x,\phi_y$ are the images of $x$ and $y$ for a transitive permutation representation $\phi$ of $G$ on $\U$.
\end{theorem}

\begin{figure}[h] 
\centering
\includegraphics[scale = 0.5]{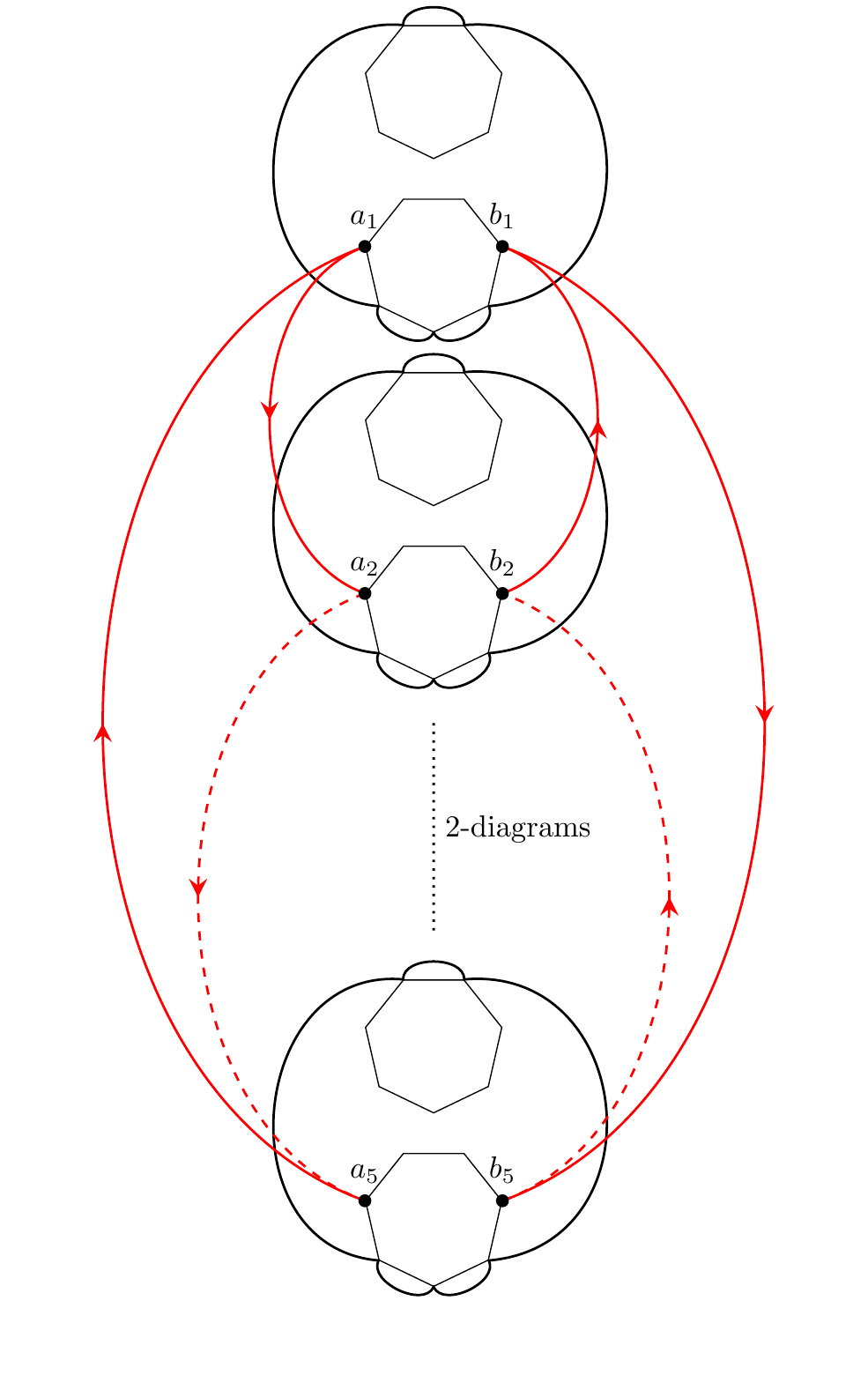}
\caption{Composition of five coset diagrams}
\label{figure:composition}
\end{figure}

\begin{proof}

We need to show that $\phi_x^{p}=1, \phi_y^{q}=1$ and $(\phi_x\phi_y)^r=1$. 
We know that $\pi_1(x)^p=1,\pi_2(x)^p=1,...,\pi_t(x)^p=1$ and $(a_1,\ldots,a_p)^p=(b_p,\ldots,b_1)^p=1$,
so clearly $\phi_x^p=1$.
Similarly $\pi_1(y)^q=\pi_2(y)^q=\ldots=\pi_t(y)^q=1$, therefore $(\phi_y)^q=1$. To finish we need to verify that $(\phi_x \phi_y)^r=1$.
Now consider the cycle of $\phi_x\phi_y$ that contains $a_i$.
Suppose that the cycle of $\pi_{j_i}(x)\pi_{j_i}(y)$ that contains $a_i$ has length $s$ such that $s|r$. Then it also contains $b_i$ and satisfies the following equations $a_i^{\pi_{j_i}(x)}=a_i$, $b_i^{\pi_{j_i}(x)}=b_i$, $a_i^{\pi_{j_i}(xy)^k}=b_i$ and $b_i^{\pi_{j_i}(xy)^{s-k}}=a_i$.
Now we have 
\begin{equation}
a_i\xrightarrow{\phi_x\phi_y}a_{i+1}^{\phi_y}\xrightarrow{(\phi_x\phi_y)^{k-1}}b_{i+1}\xrightarrow{\phi_x\phi_y}b_i^{\phi_y}\xrightarrow{(\phi_x\phi_y)^{s-(k+1)}}a_i.
\end{equation}
This implies
\begin{equation}
a_i\xrightarrow{(\phi_x\phi_y)^{1+k-1+1+s-k-1}}a_i,
\end{equation}
so,
\begin{equation}
a_i\xrightarrow{(\phi_x\phi_y)^s}a_i,
\end{equation}
and we see that $a_i$ and $b_{i+1}$ are together in a cycle of length $s$ for $\phi_x\phi_y$ that contains some points from the cycle of $\pi_{j_i}(x)\pi_{j_i}(y)$ containing $a_i$ and $b_i$ and some points from the cycle of $\pi_{j_{i+1}}(x)\pi_{j_{i+1}}(y)$ containing $a_{i+1}$ and $b_{i+1}$. We see also that a cycle of $\phi_x \phi_y$ that contains no $a_i$ has the same length as a cycle of $\pi_j(x)\pi_j(y)$ for some $j$. So $(\phi_x \phi_y)^r=1$.

Now define $\phi: G \rightarrow \S(\U)$ by $\phi(x)=\phi_x$, $\phi(y)=\phi_y$ and extending multiplicatively.
We have proved that $\phi$ is a homomorphism, defining an action of $G$ on $\U$.

To prove that the action of $G$ on the set $\U$ defined by $\phi$ is transitive, we need to check that for all $z,z' \in \U$, $\exists g \in G$ with $z^{\phi(g)}=z'$. We have the following cases

$Case$ $1$: If $z,z' \in \Omega_i$ are in the same subset, then $\exists g \in G$ such that $z^{\pi_i(g)}=z'$, because $G$ is transitive on the set $\Omega_i$.
We can write \[g=x^{i_1}y^{j_1}x^{i_2}y^{j_2}...x^{i_k}y^{j_k} \]
it is possible that $i_1=0$ or $j_k=0$.
Now we define $z_0 = z$ and
\begin{alignat*}{2}
	z_1 & =w_1^{\pi_i(y^{j_1})}
	\qquad\qquad&\qquad\qquad
	w_1 &=z_0^{\pi_i(x^{i_1})}
\\
	z_2&=w_2^{\pi_i(y^{j_2})}
	&
	w_2&=z_1^{\pi_i(x^{i_2})}
\\  
	&\vdots
	&
	&\vdots
\\
	z_k &= z' = w_k^{\pi_i(y^{j_k})}
	&
	w_k&=z_{k-1}^{\pi_i(x^{i_k})} \,.
\end{alignat*}
We want to find $g'\in G$ such that $\phi(g')$ maps $z$ to $z'$ through the same points of $z_0=z,w_1,z_1,...w_k,z_k=z'$ as $\pi_i(g)$ does.
For this, we know that $\pi_i(y)$ acts on $\Omega_i$ just as $\phi(y)$ does, however, $\pi_i(x)$ and $\phi(x)$ do not act the same on the set $\Omega_i$.
In fact, for each $\ell$, $w_\ell^{\pi_i(y^{j_\ell})}=w_\ell^{\phi(y^{j_\ell})}$, however, $z_\ell^{\pi_i(x^{i_\ell})}=z_\ell^{\phi(x^{i_\ell})}$ unless $z = a_i$ or $z = b_i$.
If $z=a_i$ or $b_i$ then $z_\ell^{\pi_i(x^{i_\ell})}=z_\ell$. So in that case $w_{\ell+1}=z_\ell$.
We form $g'$ from $g$ by deleting from $g$ all those $x^{i_\ell}$ for which $z_\ell=a_i$ or $b_i$. Then we have $z^{\phi(g')}=z'$. 

$Case$ $2$:
If $z \in \Omega_i$ and $z' \in \Omega_j$, where $i \neq j$ then $\phi(x^{j-i})$ maps $a_i$ to $a_j$. Then $\exists g_1,g_2,g_3 \in G$ such that $z^{\phi(g_1)}=a_i$, $a_i^{\phi(g_2)}=a_j$, $a_j^{\phi(g_3)}=z'$. 
This implies $z^{\phi(g_1)\phi(g_2) \phi(g_3)}=z'$. i.e $z^{\phi(g_1g_2g_3)}=z'$.
\end{proof}

\section{Imprimitive composition}
\label{sec:ImprimComp}

\begin{theorem}\label{thm:imprim_t=p}
Let $G,\pi_1,\ldots,\pi_p$ be as in Theorem~\ref{thm:comp}, suppose that $p$ is prime, and let $t=p$.
Suppose also that, for some finite set $\Omega$ of size $\deg$ and a permutation representation
$\pi$ of $G$ on $\Omega$, each $\pi_i$ is equivalent to $\pi$, via a bijection $f_i$.
Suppose further that each $(a_i,b_i)$ is a handle of $\Omega_i$, the image
of a handle $(a,b)$ of $\Omega$.

Now, for $\omega \in \Omega$, define $B_\omega \subseteq \U$ via
\[ B_\omega = \{ f_i(\omega): i = 1,\ldots p \}, \]
and let
\[\B = \{B_\omega : \omega \in \Omega \}. \]
Then the action $\psi$ of the permutation group $H:=\phi(G)$ on $\B$ is equivalent to the action of $\pi(G)$ on $\Omega$. $H$ acts imprimitively on $\U$ with blocks of imprimitivity $B_\omega$.

Let $\psi, N, J_i,K_i,Q_i,B_i$ be as defined in Proposition \ref{lemma:ImprimWreathProduct}. Then $Q$ is cyclic of order $p$ and $N$ is elementary abelian of order at most $p^\deg$.
Then $H$ is isomorphic to a subgroup of $C_p \wr \psi(H)$ and the action of $H$ on $N$ by conjugation induces an action of $\psi(H)$. Under this action $N$ is a submodule of the $\deg$-dimensional permutation module over $F_p$ for the subgroup $\psi(H)$
of $\S(\Omega)$.
\end{theorem}

\begin{proof}
For $\omega \in \Omega$, for all $g \in G$, and for $i =1,..,p$, we have
\begin{equation} \label{eq:myEqu}
 f_i(\omega^{\pi(g)}) =  f_i(\omega)^{\pi_i(g)}.
\end{equation}
Recall that for $t=p$, we have
\begin{eqnarray*}
\phi_x &=& \pi_1(x)\cdots \pi_p(x) \circ (a_1,\ldots,a_p)(b_b,b_{p-1},\cdots b_1)\\
\phi_y &=& \pi_1(y)\cdots \pi_p(y). \\
\end{eqnarray*}

We want to prove that the action of $H=\phi(G)$ on $\B$ is equivalent to the action of $\pi(G)$ on $\Omega$. We need a bijection $F:\Omega \rightarrow \B$ so that for all $\omega \in \Omega$, all $g \in G$
\[F(\omega^{\pi(g)})=F(\omega)^{\phi(g)}. \]

We define $F:\Omega \rightarrow \B$ by $F(\omega)=B_{\omega}$. We need to check that $F$ is a bijection. Clearly it is surjective. Now, 
\[ F(\omega) = F(\omega')
\Rightarrow
 \{f_1(\omega) \ldots f_p(\omega) \}
 =
 \{f_1(\omega') \ldots f_p(\omega')\}. \]
Since $f_i(\omega),f_i(\omega') \in \Omega_i$ and the sets $\Omega_i$'s are disjoint, this implies that $f_i(\omega)=f_i(\omega')$  for each $i$.
So since each $f_i$ is a bijection, we get $\omega =\omega'$.

It remains to check that $F(\omega^{\pi(g)})=F(\omega)^{\phi(g)}$ i.e. $B_{\omega^{\pi(g)}}=(B_{\omega})^{\phi(g)}$ for all $g \in G$, $\omega \in \Omega$.
First suppose that $g=y$, for all $\omega$, the image of $B_{\omega}$ under $\phi(y)$ is 
\begin{align*}
B_{\omega}^{\phi(y)}& = \{f_1(\omega)^{\pi_1(y)},f_2(\omega)^{\pi_2(y)},...,f_p(\omega)^{\pi_p(y)} \} \\
& = \{f_1(\omega^{\pi(y)}),f_2(\omega^{\pi(y)}),...,f_p(\omega^{\pi(y)}) \} \\        
 & = B_{{\omega}^{\pi(y)}}.
\end{align*} 
Now suppose that $g=x$
if $\omega \neq a,b$ we have $B_\omega^{\phi(x)}=B_{\omega^{\pi(x)}}$.  Finally 
\begin{align*}
B_{a}^{\phi(x)} & = \{a_1,a_2,...,a_p \}^{\phi(x)} \\
 & = \{a_2,a_3,...,a_p,a_1 \} \\       
 & = B_a = B_{a^{\pi(x)}}
\end{align*} 
and 
\begin{align*}
B_{b}^{\phi(x)} & = \{b_p,b_1,\ldots, b_{p-1}\}= B_b = B_{b^{\pi(x)}}.
\end{align*}
So for all $\omega$, $B_{\omega}^{\phi(x)}=B_{\omega^{\pi(x)}}$.
Since $G$ is generated by $x$ and $y$ this proves that
$B_{\omega}^{\phi(g)}=B_{\omega^{\pi(g)}}$, and hence $F$ is an equivalence.

Now the sets $B_{\omega}$ are blocks of imprimitivity for $\phi$ if and only if for each $\omega \in \Omega$, $g \in G$ either $B_{\omega}=B_{\omega^{\phi(g)}}$ or $B_{\omega} \cap B_{\omega^{\phi(g)}}=\emptyset$. 

So suppose that $B_{\omega}\cap B_{\omega}^{\phi(g)}\neq \emptyset$. Then, since $B_{\omega}^{\phi(g)}=B_{\omega^{\pi(g)}}$ we have $f_i(\omega)=f_j(\omega^{\pi(g)})$ for some $i,j$. Since $\Omega_i \cap \Omega_j = \emptyset$, we have $i=j$. Then since $f_i$ is a bijection we have $\omega=\omega^{\pi(g)}$, and $B_{\omega}=B_{\omega^{\pi(g)}}=B_{\omega}^{\phi(g)}$.

Let $J_i,K_i,P_i, Q_i$ be as defined in Proposition \ref{lemma:ImprimWreathProduct}. 
By Proposition \ref{lemma:ImprimWreathProduct}, $H$ is isomorphic to a subgroup of 
$Q \wr \psi(H)$ and $N$ to a subgroup of $P_1P_2 \cdots P_\deg$, where $Q \cong Q_i$ for each $i$.

In order to find $N$ we need to identify the groups $Q_i$ and $P_i$.

$J_a$ is the subgroup of $H$ that fixes $B_a$ setwise. So $\phi(x) \in J_a$, since $\phi(x)$ fixes the block $B_a$. In fact $\phi(x)$ permutes the points of 
$B_a$ in a $p$-cycle. We claim that for any $g \in G$, for any blocks 
$B_c=\{c_1,\ldots, c_p\}$, 
$B_{c'}=\{c'_1,\ldots, c'_p\}$, 
if $c_i^{\phi(g)} = c'_j$ then  $c_{i+k}^{\phi(g)} = c'_{j+k}$. 

We justify our claim by examining the actions of the generators $\phi(x),\phi(y)$ on the union $\U$ of the blocks.
For $g=x$, $\phi(x)$ acts on  $B_a, B_b$ via $a_i^{\phi(x)}=a_{i+1}$ and $a_{i+1}^{\phi(x)} = a_{i+2}$ and $b_p^{\phi(x)} = b_{p-1}$ and $b_{p-1}^{\phi(x)} = b_{p-2}$. Otherwise it maps $c_i\in B_c$ to some $c'_i \in B_{c'}$.
For $g=y$, $\phi(y)$ maps each $c_i$ to some $c'_i$. 

So if $\phi(g) \in J_a$ then it preserves the cyclic order of $B_a$. So $J_a/K_a \cong Q_a$ is contained in a cyclic group of order $p$. 
Then since $\phi(x)$ acts on $B_a$ as an element of order $p$, we see that 
$J_a/K_a$ contains the cyclic group of order $p$; hence $J_a/K_a \cong C_p$.
It follows from transitivity that for each $i$, we have $Q_i \cong J_i/K_i \cong C_p$.

Now $P_i \subseteq Q_i$.
So $N \subseteq P_1 \cdots P_\deg \subseteq Q_1\cdots Q_\deg=Q^\deg$. Hence $N$ is at most $C_p^\deg$. So it is elementary abelian of order at most $p^\deg$.

Now we consider the action of $H$ on $N$ by conjugation. Since $N$ is abelian, $N$ is in the kernel of this action and so there is an induced action of $\psi(H)$ on $N$. It follows from the decription of the wreath product that $Q^\deg$ is the permutation module for $\psi(H)$. So $N\subseteq Q^\deg$ must be a submodule of that permutation module. 

\end{proof}

\begin{theorem}
\label{thm:imprim_t=1}
Suppose that $G=\Delta(p,q,r)$ has a permutation 
representation $\pi$ on a finite set $\Omega$, and let $h_1=\iota,h_2,\ldots h_p:\Omega \rightarrow \Omega$ be permutations of $\Omega$
that commute with $\pi$, that is, they satisfy
\begin{equation} \label{eq:myeq2}
h_i(\omega^{\pi(g)}) = (h_i(\omega))^{\pi(g)}, \forall \omega \in \Omega, g \in G.
\end{equation}

Now suppose that  $(a,b)$ is a $k$-handle for $\pi$, and for each $i$,
define $a_i=h_i(a),b_i=h_i(b)$. Suppose that the points $a_1,b_1,a_2,b_2,...,a_p,b_p$ are all distinct. 
Define $\phi: G \rightarrow \S(\Omega)$ as in Theorem~\ref{thm:comp}.

Let
\[ B_\omega = \{ h_i(\omega): i = 1,\ldots p \} \] and let \[ \B =\{B_{\omega} :\omega \in \Omega\}. \]
Then the action of $\phi(G)$ on $\B$ is equivalent to the action of $\pi(G)$ on $\B$. And the sets $B_\omega$ are blocks of imprimitivity for the action of $\phi(G)$ on $\Omega$ if and only if they are blocks of imprimitivity for the action of $\pi(G)$ on $\Omega$.
\end{theorem}

\begin{proof}
Our conditions ensure that for $i \neq i'$, the handles $[a_i,b_i]$ and $[a_i',b_i']$ are always disjoint. As in the proof of Theorem \ref{thm:imprim_t=p}, we see that the image of ${B_\omega}$ under $\phi(y)$ is 
\begin{align*}
{B_\omega}^{\phi(y)}&=\{ h_i(\omega)^{\pi(y)}: i=1\ldots p\} \\
& =\{ h_i(\omega^{\pi(y)}) : i=1 \ldots p\} =B_{\omega^{\pi(y)}}=B_{\omega}^{\pi(y)}.
\end{align*}
Similarly \[ B_{\omega}^{\phi(x)}=B_{\omega^{\pi(x)}}=B_{\omega}^{\pi(x)},\]
so \[ B_{\omega}^{\phi(g)}=B_{\omega^{\pi(g)}}=B_{\omega}^{\pi(g)}.\]

The sets $B_{\omega}$ are blocks for $\phi(G)$ acting on $\Omega$ if and only if $B_{\omega} \cap B_{\omega}^{\phi(g)}= \emptyset$ whenever $B_{\omega}^{\phi(g)} \neq B_{\omega}$, and they are blocks for $\pi(G)$ acting on $\Omega$ if and only if $B_{\omega} \cap B_{\omega}^{\pi(g)}=\emptyset$ whenever $B_{\omega}^{\pi(g)} \neq B_{\omega}$. So since $B_{\omega}^{\pi(g)}=B_{\omega}^{\phi(g)}$, the $B_{\omega}$ are blocks of imprimitivity for $\phi(G)$ if and only if they are blocks of imprimitivity for $\pi(G)$.

\end{proof}

\begin{theorem}
\label{thm:imprim_comp}
Suppose that $G=\Delta(p,q,r)$ has a transitive permutation
representation $\pi: G \rightarrow \S(\Omega)$ of degree $\deg$ with disjoint $t$-handles
$(a,b)$ and $(c,d)$, for some $t$.

Let $m$ be an integer, and suppose that 
$\alpha=\alpha_1\cdots \alpha_k$ and $\beta=\beta_1\cdots \beta_l$ are two
permutations, both products of disjoint $p$-cycles (the $\alpha_i$ and $\beta_j$), that generate a transitive
subgroup of $\S_m$.

Then we can make a transitive permutation representation $\phi$ of $G$ of degree
$m \deg$ as follows.

Suppose that $\pi_1=\pi$ and that $\pi_2,\ldots,\pi_m$ are representations
equivalent to $\pi$ on  $\{\deg +1,\ldots,2\deg\}$, $\{2\deg+1,,\ldots,3\deg\}$,$\ldots,\{(m-1)\deg+1,\ldots,m\deg\}$, and let $(a_i,b_i)$, $(c_i,d_i)$ be the copies of the handles
$(a,b)$ and $(c,d)$ in $\pi_i$.

Then for each of the cycles $\alpha_i=(i_1,\ldots,i_p)$ of $\alpha$,
we define \[\gamma_i=(a_{i_1},\ldots,a_{i_p})(b_{i_p},\ldots,b_{i_1}),\]
and for each of the cycles 
$\beta_j=(j_1,\ldots,j_p)$ of $\beta$,
we define \[\delta_j=(c_{j_1},\ldots,c_{j_p})(d_{j_p},\ldots,d_{j_1}).\]

Then we define
\begin{eqnarray*}
\phi(x)&=&\pi_1(x)\pi_2(x)\cdots \pi_m(x)\gamma_1\gamma_2\cdots\gamma_k\delta_1\delta_2\cdots\delta_l\\
\phi(y)&=&\pi_1(y)\pi_2(y)\cdots \pi_m(y).
\end{eqnarray*}

Let \[\B=\{ B_{\omega} : \omega \in \Omega \}. \] 
The action $\psi$ of $H := \phi(G)$ on $\B$ is equivalent to the action of $\pi(G)$ on $\Omega$.

The representation $\phi$ is imprimitive with blocks 
\[ B_\omega = \{ \omega,\deg+\omega,2\deg+\omega,\ldots, (m-1)\deg+\omega \}, \]
for each $\omega \in \Omega$.
Let $J_i,K_i,P_i,Q_i$ be as defined in Proposition \ref{lemma:ImprimWreathProduct}. 
Then $Q_i \cong \langle \alpha,\beta \rangle$.
Hence $H$ is isomorphic to a subgroup of $\langle \alpha,\beta \rangle \wr \psi(H)$.
\end{theorem}

\begin{proof}
Since we could construct $\phi$ by repeating the construction of Theorem $\ref{thm:comp}$, it is clear we have a permutation representation $\phi$ of degree $m \deg $. To see that the sets $B_\omega$ are blocks, we consider their images under $\phi(y)$ and $\phi(x)$.
We see that 

\[B_{\omega}^{\phi(y)} = \{\omega^{\pi(y)},\deg+\omega^{\pi(y)},...,(m-1)\deg+\omega^{\pi(y)} \}=B_{\omega^{\pi(y)}} \]

Similarly, if $\omega \neq a,b,c,d$ then image of $B_{\omega}$ under $\phi(x)$ is

\[B_{\omega}^{\phi(x)} = \{\omega^{\pi(x)},\deg+\omega^{\pi(x)},...,(m-1)\deg+\omega^{\pi(x)} \}= B_{\omega^{\pi(x)}}.\]
Finally, for $\omega=a,b,c,d$ we have $B_{a}^{\phi(x)}= B_{a}= B_{a^{\pi(x)}}$, $B_{b}^{\phi(x)}= B_{b}= B_{b^{\pi(x)}}$, $B_{c}^{\phi(x)}= B_{c}= B_{c^{\pi(x)}}$ and $B_{d}^{\phi(x)}= B_{d}= B_{d^{\pi(x)}}$.

So for all $\omega$, $g \in G$ $B_{\omega}^{\phi(g)}=B_{\omega^{\pi(g)}}$. This proves that the action of $\phi(G)$ of $\B$ is equivalent to the action of $\pi(G)$ on set $\Omega$.
Then just as in Theorem \ref{thm:imprim_t=p} we can see that $B_{\omega}$ are blocks of imprimitivity for the image of $G$ under $\phi$.

To prove that $Q_a \cong\langle \alpha, \beta \rangle$, we   
prove first that $Q_a \subseteq \langle \alpha, \beta \rangle$
and then that $Q_a \supseteq \langle \alpha, \beta \rangle$.

To prove that $Q_a \subseteq \langle \alpha, \beta \rangle$, we need to show that for any $g \in G$ with $\phi(g) \in J_a$, $\phi(g)$ acts on $B_a$ as an element of $\langle \alpha, \beta \rangle$. 

Let $g= x^{i_1} y^{j_1} \ldots x^{i_k} y^{j_k}$, where $\phi(g)= \phi(x)^{i_1}\phi(y)^{j_1} \ldots \phi(x)^{i_r}\phi(y)^{j_k} \in J_a$.
Define $z_0,z_1,\ldots,z_k=a' \in \U$ with $a' \in B_a$, by
$$a=z_0\xrightarrow{\phi(x)^{i_1}\phi(y)^{j_1}}z_1\xrightarrow{\phi(x)^{i_2}\phi(y)^{j_2}}z_2 \ldots z_{k-1} \xrightarrow{\phi(x)^{i_r}\phi(y)^{j_r}}z_k=a.$$

Let $a_1,\ldots,a_\deg$ be the points of $B_a$. 
Where the points of $\Omega$ are labelled $1,2,\ldots,\deg$,
the construction ensures that
the point $a_\ell$ of $B_a \subseteq \U$ is numbered $a+(\ell-1)\deg$.
So now suppose that 
$a_\ell= a+(\ell -1)\deg$ is an arbitrary point of $B_a$ for some $\ell \in\{1 \ldots m\}$
We have $$ (a+(\ell -1)\deg)^{\phi(g)} = a+(\ell^{\xi^{i_1}\xi^{i_2} \ldots \xi^{i_r}}-1)\deg$$
where each $\xi_i=1,\alpha,\alpha^{-1},\beta,\beta^{-1}$ such that
$$ 
\xi_i=
\begin{cases}
1 & \text{if } \,z_i \not\in B_a \cup B_b \cup B_c \cup B_d,\\
\alpha & \text{if }\, z_i \in B_a,\\
\alpha^{-1}& \text{if }\, z_i \in B_b,\\
\beta & \text{if } \,z_i  \in B_c,\\
\beta^{-1} & \text{if }\, z_i \in B_d.
\end{cases}
$$
So $\phi(g)$ acts on $B_a$ as $\xi^{i_1}\xi^{i_2} \ldots \xi^{i_r} \in \langle \alpha, \beta \rangle$. 
This shows that $Q_a \subseteq \langle \alpha, \beta \rangle$.

We see easily that $\phi(x) \in J_a \cap J_b \cap J_c \cap J_d$,
and that the element $\phi(x)$ permutes the
points of $B_a$ in the same
way that $\alpha$ permutes the points of $\{1,\ldots,m\}$. 

To complete the proof that $\langle \alpha,\beta \rangle \subseteq Q_a$ we need to find a conjugate of $\phi(x)$,
$\phi(g)\phi(x)\phi(g)^{-1}$, that acts on $B_a$ as $\beta$. For this, we choose $g \in G$ to be a shortest possible word in $x$ and $y$ such that $\pi(g)$
(acting on $\Omega= \{ 1,2,\ldots,\deg\}$) takes 
$a$ to $c$. 
Define $z_0,z_1,\ldots,z_k \in \Omega$ such that
$$a=z_0\xrightarrow{\pi(x^{i_1})\pi(y^{j_1})}z_1 \xrightarrow{\pi(x^{i_2})\pi(y^{j_2})}z_2 \ldots z_{k-1}\xrightarrow{\pi(x^{i_k})\pi(y^{j_k})}z_k=c.$$

Suppose $z_j=a,b,c,d$ for $j<k$. Then $z_j^{\pi(x)}=z_j$ and $z_j^{\pi(x^{i_{j+1}})}=z_j$, so we can leave out $x^{i_{j+1}}$ and get a shorter choice for $g$. 
This cannot happen, so we can assume
that $z_j \notin \{a,b,c,d \}$ for $j<k$. 
We choose $i_1=0$ since $a^{\pi(x)}=a$. So
$(a+(\ell-1)\deg)^{\phi(y)^{j_1}}=z_1+(\ell-1)\deg$.

$z_1 \neq \{a,b,c,d \}$ by assumption as above. So 
$(z_1+(\ell-1)\deg)^{\phi(x)^{i_2}}=z_1^{'}+(\ell-1)\deg$ for some $z_1^{'} \neq z_1$, and 
$(z_1^{'}+(\ell-a)\deg)^{\phi(y)^{j_2}}=z_2+(\ell-1)\deg$. 

$z_2 \not \in \{a,b,c,d \}$ provided that $2 < k$,
and we continue as above.
 
We end up with, for each $\ell$,
\begin{align*}
(a+(\ell-1)\deg)^{\phi(g)} &= c+(\ell-1)\deg \\
(a+(\ell^\beta-1)\deg)^{\phi(g)} &= c+(\ell^\beta-1)\deg \\
(a+(\ell-1)\deg)^{\phi(g)\phi(x)} &= c+(\ell^{\beta}-1)\deg\\
(a+(\ell-1)\deg)^{\phi(g)\phi(x)\phi(g)^{-1}} &= a+(\ell^\beta-1)\deg
\end{align*}
So $\phi(g)\phi(x)\phi(g)^{-1}$ acts on $B_a$ as $\beta$. This shows that $\langle \alpha, \beta \rangle \subseteq Q_a$. Hence, $Q_a \cong \langle \alpha, \beta \rangle$. To finish we apply Proposition \ref{lemma:ImprimWreathProduct}.

\end{proof}

\section{Imprimitive composition with alternating groups}
\label{sec:Alt}

\begin{theorem}
\label{thm:ImprimCompAlt1}
Suppose that $G=\Delta(p,q,r)$ is a triangle group with $p$ prime, $p\leq q \leq r$. Suppose that $\pi(G) \cong \A_\deg$ and $H=\phi(G)$ is constructed as in Theorem \ref{thm:imprim_t=p}. Assume the notation of Theorem \ref{thm:imprim_t=p}.
Suppose that  $\deg > 6$.
Then either 

\begin{enumerate}
\item $p|qr$, $p | \deg$ and $H \cong C_p \times \A_\deg$,
or \item $H \cong C_p^{\deg -1} \rtimes \A_\deg$.
\end{enumerate}
Note that the second case might occur even when $p|qr$ and $p|\deg$.
\end{theorem}

\begin{proof}
By Theorem~\ref{thm:imprim_t=p},
$H$ is isomorphic to a subgroup of $C_p \wr \psi(H) \cong C_p^\deg \rtimes
\psi(H)$. And since $\psi(H) \cong \pi(G) \cong \A_\deg$, we have
$H \subseteq C_p^\deg \rtimes \A_\deg$.
In fact $H \cong N . \A_\deg$.

In addition we know from Theorem~\ref{thm:imprim_t=p}
that $N \subseteq C_p^\deg$ is a submodule of the $\deg$-dimensional permutation module $W_\deg$ for $\A_\deg$. It is possible that we have $N=1$.

We hope to prove that, 
except in the situation which can only occur for the particular values of $p,d$ covered in (1),
the normal subgroup $N$ is isomorphic to $C_p^{\deg-1}$. 

First we show that $N \neq 1$. 
If $N=1$, then $H \cong \psi(H) =\A_\deg$. In that case the group $J_a$, (which stabilises $B_a$ as a set) is isomorphic to $\A_{\deg-1}$. 
However, by Theorem~\ref{thm:imprim_t=p} $J_a$ acts on $B_a$ as $Q_a$, which is cyclic of order $p$. So there is a homomorphism from $J_a$ to $C_p$, i.e. there is a homomorphism from $\A_{\deg-1}$ to $C_p$. Now the kernel of this homomorphism is a normal subgroup of $\A_{\deg-1}$ of index $p$. Since if $\deg-1 \geq 5$, the group $\A_{\deg-1}$ is simple, but this cannot happen. 

Now we need to show that $N$ is a submodule of $W_{\deg-1}$. 
Examining the construction of $H=\phi(G)$ we see that $H$ is generated by 
$$ \phi_x = \pi_1(x)\pi_2(x) \ldots \pi_p(x) \tau_a \tau_b^{-1} $$  where $\tau_\omega=(\omega_1,\omega_2, \ldots, \omega_p)$ and 
$$ \phi_y=\pi_1(y)\pi_2(y) \ldots \pi_p(y). $$
Now define $$C=\{\pi_1(g)\pi_2(g) \ldots \pi_p(g) : g \in G \}$$ and 
$$ V= \langle \tau_\omega \tau_{\omega'}^{-1} : \omega, \omega' \in \Omega \rangle $$
both subgroups of $\S(\U)$. We see that $V$ is isomorphic to the $(\deg-1)$-dimensional submodule $W_{\deg-1}$ of $Q^\deg$ for $\A_\deg$, and $C$ is isomorphic to $\pi(G)\cong \A_\deg$. Then $\phi_x,\phi_y$ can both be written as products within $CV$, so $H$ which is generated by $\phi_x$ and $\phi_y$ is a subgroup of $W_{\deg-1} \rtimes \A_\deg$. In particular $N \subseteq W_{\deg-1}$, so $|N| \leq p^{\deg-1}$. Hence $N$ is a non-trivial submodule of $W_{\deg-1}$. 

$W_{\deg-1}$ has no proper submodules unless $p$ divides $\deg$. If $p | \deg$ 
then the trivial permutation module $W_1$ is a submodule of $W_{\deg-1}$. So if $p$ does not divide $\deg$ we see that we must have $N=W_{\deg-1}$.

So suppose now that $p | \deg$, that $N$ is the trivial submodule $W_1$ of $W_{\deg-1}$, and so $H=N.\A_\deg=C_p.\A_\deg$. We look at $J_a$, the stabiliser of $B_a$. We have $\psi(J_a)=\A_{\deg-1}$. We also have $Q_a \cong C_p$. So $J_a$ has a quotient isomorphic to $C_p$. 
So $J_a$ maps onto $C_p$ and also maps onto $\A_{\deg-1}$. By the Jordan-H\"{o}lder Theorem, $J_a\cong C_p.\A_{\deg-1}$ where $C_p$ and $\A_{\deg-1}$ are simple and their intersection is the identity subgroup. So $J_a$ must be split. 

We see that $H$ must also be a split extension in this case, because the non-split extensions are well known by \cite{Schur} (and see \cite{Wilson}). The non-split extension $C_p.\A_\deg$ can only exist when $p=2$ and cannot admit such subgroups $J_a$ as above.  

So the extension $N.\A_\deg$ splits, and then, since the action of $\A_\deg$ on $N$ 
is trivial, we must have a direct product, $H=C_p \times \A_\deg$, and so are in case (1).
In that case $C_p$ is a homomorphic image of $G= \Delta(p,q,r)$, and hence an
abelian quotient of it. But a straightforward computation shows that
$G/G'$ can only map onto $C_p$ when $p|qr$.

If $N=W_{\deg-1}$, then $H$ is an extension of $C_p^{\deg-1}$ by $\A_{\deg}$. So now $H$ is a subgroup of $W_{\deg-1}\rtimes \A_\deg$ of order $p^{\deg-1}.|\A_\deg|$. Hence $H$ is the whole of $W_{\deg-1}\rtimes \A_\deg$ and the result is proved.  

\end{proof}
The following is an immediate corollary of the theorem.
\begin{corollary}
Suppose that $G=\Delta(p,q,r)$ is a triangle group with $p$ prime, $p \leq q \leq r$.
Suppose that $\deg > 6$ and in addition $p \nmid qr$ and $p \nmid \deg$.
Then provided that $G$ has a quotient $\A_\deg$ containing at least one handle, $G$ also has a quotient $C_p^{\deg-1} \rtimes \A_\deg$.
\end{corollary}

\begin{theorem}
\label{thm:ImprimCompAlt2}
Suppose that $G=\Delta(p,q,r)$ is a triangle group with $p$ prime, $p\leq q \leq r$. Suppose that $\pi(G) \cong \A_\deg$ and $H=\phi(G)$ is constructed as in Theorem \ref{thm:imprim_comp}. Assume the notation of Theorem \ref{thm:imprim_comp}.
Suppose that $p$ is prime, $\deg > 6$ and $\langle \alpha, \beta \rangle \cong \A_m$, where $m \neq \deg -1 $ and $m \geq 5$. Then $H \cong \A_m \wr \A_\deg$.
\end{theorem}

\begin{proof}
By Theorem~\ref{thm:imprim_comp},
$H$ is isomorphic to a subgroup of $\A_m \wr \psi(H) \cong \A_m^\deg \rtimes \A_\deg$.
Then $N$, the kernel of the map $\psi:H \rightarrow \A_\deg$ is a subgroup of $\A_m^\deg$ and $H=N.\A_\deg$.
In order to prove the result, we need simply to prove that $N$ is the whole of $\A_m^\deg$.

We have $N \subseteq T_1 \times T_2 \times \cdots \times T_\deg$, where $T_i \cong \A_m$ for each $i$. 

By \cite[Lemma 1.4.1]{Fawcett}, $N$ is a direct product of groups $H_1 \cdots H_r$ where each $H_i$ is a full diagonal subgroup of $\prod_{i \in I_j}T_i$, and $I_1, \ldots, I_r$ is a partition of $\{1,\ldots,\deg \}$. Now the partition must be preserved by $\psi(H)=\A_\deg$, in its action on $N$ by conjugation. Since $\A_\deg$ acts primitively, so either we have $r=1$ and $I_1=\{1,\ldots,\deg \}$ or we have $r=\deg$ and $I_j=\{ j\}$ for each $j$. 

In the first case we have $N=\A_m$ and in the second case we have $$N=\underbrace{\A_m \times \A_m \times \cdots \times \A_m}_\text{\deg \hspace{0.2cm} times}.$$ When $N=\A_m$ then we have $H \cong \A_m.\A_\deg$. 

Considering the subgroup $J_a$, which maps onto $\A_m$, we see that in this case $H$ must be the direct product $\A_m \times \A_\deg$. 
So then $H$ is the direct product of $N$ $(\cong \A_m)$ and its complement 
\[C= \{ \pi_1(g), \pi_2(g), \ldots, \pi_m(g): g \in G\}. \]

Then every element of $H$ can be written as a product $nc$ where $n \in N$, $c \in C$ and the elements $n,c$ commute. Now $\phi(x)=c_1n_1$, where 
\begin{align*}
c_1 &= \pi_1(x)\pi_2(x) \cdots \pi_m(x)\\
n_1 &= \gamma_1 \cdots \gamma_k\delta_1 \cdots \delta_l
\end{align*}
and $\phi(y) \in C$. Since $H$ is generated by  $\phi(x)$ and $\phi(y)$, we see that $N$ must be cyclic, generated by $\gamma_1 \cdots \gamma_k \delta_1 \cdots \delta_l$, and hence $N$ is cyclic of order $p$. 

This contradicts the fact that $N \cong \A_m$ for $m \geq 3$ and so this case is excluded. 
\end{proof}
%
The following is an immediate corollary.
\begin{corollary}
Suppose that $G=\Delta(p,q,r)$ is a triangle group with $p$ prime, $p \leq q \leq r$.
Suppose that  $\deg > 6$, and that for some $m$ not equal to $\deg-1$ the alternating group $\A_m$ can be generated by two $p$-cycles.
Then provided that $G$ has a quotient $\A_\deg$ containing two disjoint handles, $G$ also has a quotient
$\A_m \wr \A_\deg$.
\end{corollary}
Our results suggest the following two conjectures.
\begin{conjecture}
Suppose that $G=\Delta(p,q,r)$ is a triangle group with $p$ prime, $p\leq q \leq r $. Then for all but finitely many integers $\deg$, $G$ maps on to $C_p^{\deg-1} \rtimes \A_\deg$.
\end{conjecture}

\begin{conjecture}
Suppose that $G=\Delta(p,q,r)$ is a triangle group with $p$ prime $p \leq q \leq r$ and choose $m$ such that the alternating group $\A_m$ can be generated by 2 $p$-cycles. Then for all but finitely many integers $\deg$, $G$ maps on to $\A_m \wr \A_\deg$. 
\end{conjecture}

\begin{acknowledgements}
My deepest gratitude to my supervisor Prof. Sarah Rees and Prof. Derek Holt for their constant guidance during my PhD. This research is supported by Faculty for the Future, Schlumberger Foundation. 
\end{acknowledgements}

\vfill\eject

  \bibliographystyle{plain}
  \bibliography{references}

\end{document}